\documentclass{daj}

\dajAUTHORdetails{%
  title = {Simple Proofs for Furstenberg Sets Over Finite Fields}, 
  author = {Manik Dhar, Zeev Dvir, and Ben Lund},
  plaintextauthor = {Manik Dhar, Zeev Dvir, and Ben Lund},
  keywords = {Kakeya, Furstenberg, finite fields, min-entropy, polynomial method},
}   

\dajEDITORdetails{%
   year={2021},
   number={22},
   received={19 September 2019},   
   revised={12 June 2021},    
   published={14 October 2021},  
   doi={10.19086/da.29067},       
}   

\usepackage{verbatim}
\usepackage{amsmath, amssymb, amsfonts, amsthm}
\usepackage{setspace}
\usepackage{algorithmic}
\usepackage{enumitem}
\usepackage{tabularx}
\usepackage{thmtools}
\usepackage{thm-restate}
\usepackage{sidecap}
\sidecaptionvpos{figure}{c}

\setlist{nolistsep}
\def\minentq{{\mathbf H}^{q}_{\infty}}

\def \F{\mathbb{F}}

\def \Z{\mathbb{Z}}
\def \bL{\mathbf{L}}
\def \bH{\mathbf{H}}
\def \bP{\mathbf{Pr}}
\def \PFq[#1]{\mathbf{P}^#1(\mathbb{F}_q)}
\def \PF[#1,#2]{PG(#1,#2)}

\def \eps{\varepsilon}

\declaretheorem[name=Theorem]{theorem}

\newtheorem{lemma}[theorem]{Lemma}
\newtheorem{definition}[theorem]{Definition}

\newtheorem{claim}[theorem]{Claim}
\newcommand\numberthis{\addtocounter{equation}{1}\tag{\theequation}}

\begin{document}

\begin{frontmatter}[classification=text]

\title{Simple Proofs for Furstenberg Sets Over Finite Fields} 

\author[dhar]{Manik Dhar}
\author[dvir]{Zeev Dvir\thanks{Supported by NSF grant DMS-1953807.}}
\author[lund]{Ben Lund\thanks{Supported by NSF postdoctoral fellowship DMS-1802787 and Institute for Basic Science (IBS-R029-C1).}}

\begin{abstract}
    A $(k,m)$-Furstenberg set $S \subset \F_q^n$  over a finite field is a set that has at least $m$ points in common  with a $k$-flat in every direction. The question of determining the smallest size of such sets is a natural generalization of the finite field Kakeya problem. The only previously known bound for these sets is due to Ellenberg-Erman \cite{ellenberg2016furstenberg} and requires sophisticated machinery from algebraic geometry. In this work we give new, completely elementary and simple proofs that significantly improve the known bounds. Our main result relies on an equivalent formulation of the problem using the notion of min-entropy, which could be of independent interest.
\end{abstract}
\end{frontmatter}

\section{Introduction}

	For a prime power $q$, let $\F_q$ be the finite field of order $q$.
	Let $n > k \geq 1$ and $m \geq 1$ be integers.
	A subset $S \subseteq \F_q^n$ is a {\em $(k,m)$-Furstenberg} set if, for each rank $k$ subspace $W$ of $\F_q^n$, there is a translate of $W$ that intersects $S$ in at least $m$ points.
	
	For a prime power $q$ and integers $n,k,$ and $m$ with $1 \leq k < n$ and $m \leq q^k$, let $K(q,n,k,m)$ be the least $t$ such that there exists a $(k,m)$-Furstenberg set in $\F_q^n$ of cardinality $t$.
	
	A $(1,q)$-Furstenberg set is called a Kakeya set.
	The question of determining $K(q,n,1,q)$ was originally posed by Wolff \cite{wolf1999} as a toy version of the Euclidean Kakeya conjecture.
	For this case, the polynomial method \cite{dvir2009size,saraf2008,dvir2013extensions} gives the bound
	
	\begin{equation}\label{eq:kakeya}
		K(q,n,1,q) \geq 2^{-n}q^n,
	\end{equation}
	which is tight up to a factor of 2.
	This was recently improved by Bukh and Chao \cite{bukh}, who proved a bound that is tight up to lower order terms.
	The same techniques also handle the more general case of arbitrary $m$, giving the bound
	
	\begin{equation}\label{eqn:generalMKakeya}
	    K(q,n,1,m) \geq 2^{-n}m^n.
	\end{equation}

	The approach used to prove (\ref{eq:kakeya}) was generalized to all $k$ when $m=q^k$ by Kopparty, Lev, Saraf, and Sudan \cite{kopparty2011kakeya}, who improved earlier work by Ellenberg, Oberlin, and Tao \cite{ellenberg_oberlin_tao_2010}. They show
	
	\begin{equation}\label{eq:fullFurstenberg}
		K(q,n,k,q^k) \geq \left( \frac{q^{k+1}}{q^k + q - 1} \right)^n = \left (1+\frac{q-1}{q^k} \right )^{-n} q^n.
	\end{equation}
	
	For fixed $k \geq 2$, fixed $n$, and $q$ large, (\ref{eq:fullFurstenberg}) states that a $(k,q^k)$-Furstenberg set in $\F_q^n$ must contain nearly all of the points of $\F_q^n$.
	For fixed $k \geq 2$, fixed $q$, and $n$ large, (\ref{eq:fullFurstenberg}) states that a $(k,q^k)$-Furstenberg set in $\F_q^n$ must have size at least $C^{-n} q^n$, for some constant $C>1$ depending on $q$ and $k$.
	
	Kopparty, Lev, Saraf, and Sudan also described several ways to construct small Furstenberg sets when $m=q^k$.
	We include only one of these here. 
	Other constructions described in \cite{kopparty2011kakeya} give better bounds for large $k$, and for some explicit, small values of $q$.
	\begin{equation}\label{eq:furstConstruction}
		K(q,n,k,q^k) \leq \left (1 - \frac{q-3}{2 q^k} \right )^{\lfloor n/(k+1) \rfloor} q^n.
	\end{equation}
	
	Furstenberg sets with $k \geq 2$ and $m < q^k$ are not understood as well.
    The first progress on the general case was by Ellenberg and Erman \cite{ellenberg2016furstenberg}, who used a sophisticated algebraic argument to prove
	
	\begin{equation}\label{eq:generalFurstenberg}
		K(q,n,k,m) \geq C_{n,k}m^{n/k}.
	\end{equation}
	
	Ellenberg and Erman did not explicitly specify the value of $C_{n,k}$ obtained, but a close inspection of the proof shows that it is $C_{n,k} = (1/n)^{\Omega(n \ln(n/k))}$.
	Recent work of the current authors \cite{DDL-1} gives a slightly more streamlined version of the Ellenberg and Erman  proof to obtain (\ref{eq:generalFurstenberg}) with $C_{n,k} = \Omega((1/16)^{n \ln (n/k)})$. 
	
	The contribution of this paper is to improve (\ref{eq:generalFurstenberg}) using much simpler and more elementary arguments. Our first main result deals with the case of general $k$ and $m \leq q^k$:
	
	\begin{theorem}\label{thm:FurstenbergRecurse}
	Let $q$ be a prime power, and let $n,k,$ and $m$ be positive integers such that $m \leq q^k$, then
$$K(q,n,k,m)\ge \frac{1}{2^n} m^{n/k}. $$
	\end{theorem}
	
	Ellenberg and Erman's method can be used to prove Furstenberg-style bounds involving hypersurfaces that don't follow from the proof of Theorem~\ref{thm:FurstenbergRecurse}. The proof of Theorem~\ref{thm:FurstenbergRecurse} relies on a new equivalent formulation of the problem using the notion of min-entropy. This new formulation, described in Section~\ref{sec:entropy}, allows us to derive the bound for general $k$ using a recursive argument, starting with $k=1$ as a base case (proved using the polynomial method).

	A separate argument gives stronger bounds for large $m$. Let $S$ be any set of $mq^{n-k}$ points in $\F_q^n$. A simple  pigeonholing argument shows that $S$ is a $(k,m)$-Furstenberg set.
	When $m$ is sufficiently large relative to $q$, it turns out that there are no Furstenberg sets much smaller than this trivial construction.
	
	\begin{theorem}\label{th:largeMFurst}
		Let $\eps > 0$, let $q$ be a prime power, and let $n,k,$ and $m$ be integers with $2 \leq k < n$ and $m \leq q^k$.
		If $m \geq 2^{n+7-k}q \eps^{-2}$, then
		\[K(q,n,k,m) \geq (1-\eps)m q^{n-k}.\]
	\end{theorem}
	
	Note that, since $q^k \geq m$, Theorem \ref{th:largeMFurst} never applies if $q^{k-1} < 2^{n+7-k}$.
	
	When $k > n/2$ and $m > q^{n-k}$, we can remove the assumption that the $k$-flats are in different directions and still prove a stronger bound than previously known. The number of rank $k$ subspaces in $\F_q^n$ is given by the $q$-binomial coefficient $\binom{n}{k}_q$ (see Section~\ref{sec:finiteGeometry} for details).
	
	\begin{theorem}\label{th:pureIncidences}
	    Let $q$ be a prime power, and let $n,k,$ and $m$ be integers with $n/2 < k < n$ and $0 \leq m \leq q^k$.
		Let $S \subseteq \F_q^n$.
		Let $L$ be a set of $k$-flats that each contain at least $m$ points of $S$, with $|L| = \binom{n}{k}_q$.
		Then, 
		\[|S| \geq \left (1-q^{n-2k} - \sqrt{q^{n-k} m^{-1}} \right ) mq^{n-k}.\]
		In particular, the same lower bound holds for $K(q,n,k,m)$.
	\end{theorem}
	
	Note that, if $m < q^{n-k}$, then the right side of the inequality in Theorem \ref{th:pureIncidences} is negative. Hence Theorem \ref{th:pureIncidences} is interesting only for larger $m$.
	
	The proof of Theorem \ref{th:largeMFurst} combines (\ref{eq:kakeya}) with incidence estimates for large sets in finite fields. The proof of Theorem \ref{th:pureIncidences} relies only on incidence estimates for large sets in finite fields, and doesn't rely on the polynomial method.

	Lastly, when $n$ is divisible by $k$, a very simple proof shows that the following  bound follows directly from (\ref{eqn:generalMKakeya}).
	
	\begin{theorem}\label{thm:FurstenbergDiv}
	Let $q$ be a prime power, and let $n,k$ and $m$ be positive integers such that $m \leq q^k$ and $n$ is divisible by $k$, we have
$$K(q,n,k,m)\ge \frac{1}{2^{n/k}} m^{n/k} .$$
	\end{theorem}
	
	\paragraph{Organization:} We begin in Section~\ref{sec:prelim} with some preliminaries on finite geometry and polynomials over finite fields. In Section~\ref{sec:entropy} we discuss the equivalent entropic formulation to the problem of bounding the size of Furstenberg sets. In Section~\ref{sec:entropykakeya} we prove the one dimensional case  of the entropic version using the polynomial method  and in Section~\ref{sec:entorpyfurst} we prove the general case (Theorem~\ref{thm:FurstenbergRecurse}) using recursion.
	Theorem~\ref{thm:FurstenbergDiv} is proved in Section~\ref{sec:divide} and Theorems \ref{th:largeMFurst} and \ref{th:pureIncidences} are proved in Section~\ref{sec:geometric}.

\section{Preliminaries}\label{sec:prelim}
\subsection{Facts from finite geometry}\label{sec:finiteGeometry}
	In this section, we review a few basic facts from finite geometry, as well as the results we need from incidence geometry.
	
	A {\em $k$-flat} is a translate of a rank $k$ linear subspace.
	The span of a set $X \subseteq \F_q^n$ is the smallest flat that contains $X$, and is denoted $\overline{X}$.
    For flats $\Lambda,\Gamma$ in $\F_q^n$, we denote by $\overline{\Lambda,\Gamma}$ the span of $\Lambda \cup \Gamma$.
    If $\Lambda$ and $\Gamma$ are subspaces ({\em i.e.} they each contain the origin), then
    \begin{equation}\label{eq:dimSpan}
        \dim(\overline{\Lambda, \Gamma}) = \dim(\Lambda) + \dim(\Gamma) - \dim(\Lambda \cap \Gamma).
    \end{equation}
	
	For integers $1 \leq k < n$, the number of rank $k$ subspaces of $\F_q^n$ is given by the $q$-binomial coefficient $\binom{n}{k}_q$.
	As with ordinary binomial coefficients, the $q$-binomial coefficients are centrally symmetric:
	\begin{equation}
		\label{eq:centrallySymmetric} \binom{n}{k}_q = \binom{n}{n-k}_q.
	\end{equation}
	The Pascal identities for $q$-binomial coefficients are
	\begin{align}
	\label{eq:pascal1} \binom{n}{k}_q &= q^k \binom{n-1}{k}_q + \binom{n-1}{k-1}_q, \text{ and} \\
	\label{eq:pascal2} \binom{n}{k}_q &= \binom{n-1}{k}_q + q^{n-k} \binom{n-1}{k-1}_q.
	\end{align}
	A direct expression is given by
	\begin{equation}\label{eq:exactQBinom} \binom{n}{k}_q = \frac{(1 - q^n)(1- q^{n-1}) \ldots (1 - q^{n-k+1})}{(1-q)(1-q^2)\ldots (1-q^k)}.\end{equation}

	The number of $k$-flats in $\F_q^n$ is $ q^{n-k}\binom{n}{k}_q$.
	
    A point is {\em incident} to a flat if the point is contained in the flat.
	Given a set $L$ of flats, and a set $S$ of points, both in $\F_q^n$, we denote by
	\[I(S,L) = | \{(p,\ell) \in S \times L : p \in \ell \} | \]
	the number of incidences between $S$ and $L$.
	
	The following bound on the number of incidences between points and $k$-flats was first proved by Haemmers \cite[Chapter 3]{haemers1980eigenvalue}.
	The exact statement used here can also be recovered from the proof of Theorem 1 in \cite{lund2016incidence}.
	
	\begin{lemma}\label{th:incidenceBound}
		If $S$ is a set of points and $L$ a set of $k$-flats, both in $\F_q^n$, then
		\[I(S,L) \leq q^{k-n} |S|\, |L| + \sqrt{q^k \binom{n-1}{k}_q |S| \, |L| \, (1-|S|q^{-n})\left(1 - |L|q^{k-n}\binom{n}{k}_q^{-1}\right)}. \]
	\end{lemma}
	
	Given a set $S$ of points, a flat is {\em $(S,t)$-rich} if it contains at least $t$ points of $S$.
	A flat is {\em $(S,t)$-poor} if it contains fewer than $t$ points of $S$.
	The following upper bound on the number of $(S,t)$-poor flats is a slight reformulation of \cite[Corollary 5]{lund2016incidence}.
	A slightly weaker bound was proved earlier by Alon \cite{alon1986eigenvalues}.
	
	\begin{lemma}\label{th:becks}
		Let $S \subset \mathbb{F}_q^k$ be a set of $m$ points.
		Let $0 < \delta < 1$ and $1 \leq \ell \leq k-1$.
		The number of $(S,\delta m q^{\ell - k} + 1)$-poor $\ell$-flats is at most 
		\[\left (1 + mq^{\ell-k}(1-\delta)^2 \right )^{-1} q^{k-\ell} \binom{k}{\ell}_q.\]
	\end{lemma}

	\subsection{Method of multiplicities}\label{prelim:multiplicities}

The results here are from a paper by Dvir, Kopparty, Saraf, and Sudan~\cite{dvir2013extensions}. We state the theorems we need and the proofs can be found in the aforementioned paper.

\begin{definition}[Hasse Derivatives]
Given a polynomial $P\in \F[x_1,\hdots,x_n]$ and an $i\in \Z_{\ge 0}^n$, the $i$th {\em Hasse derivative} of $P$ is the polynomial $P^{(i)}$ in the expansion $P(x+z)=\sum_{i\in \Z_{\ge 0}^n} P^{(i)}(x)z^i$ where $x=(x_1,...,x_n)$, $z=(z_1,...,z_n)$ and $z^i=\prod_{j=1}^n z_j^{i_j}$.  
\end{definition}

Hasse derivatives satisfy some useful identities. We state the only one we will need.

\begin{lemma}\label{lem:chainRule}
Given a polynomial $P\in \F[x_1,\hdots,x_n]$ and $i,j\in \Z_{\ge 0}^n$, we have 
$$(P^{(i)})^{(j)}=P^{(i+j)}\prod\limits_{k=1}^n\binom{i_k+j_k}{i_k}$$
\end{lemma}

We make precise what it means for a polynomial to vanish on a point $a\in \F^n$ with multiplicity. First we recall for a point $j$ in the non-negative lattice $\Z^n_{\ge 0}$, its weight is defined as $\text{wt}(i)=\sum_{i=1}^n j_i$.

\begin{definition}[Multiplicity]
For a polynomial $P\in \mathbb{F}[x_1, \ldots, x_n]$ and a point $a\in \mathbb{F}^n$, we say $P$ vanishes on $a$ with {\em multiplicity} $N$, if $N$ is the largest integer such that all Hasse derivatives of $P$ of weight strictly less than $N$ vanish on $a$. We use $\text{mult}(P,a)$ to refer to the multiplicity of $P$ at $a$.
\end{definition}

Notice, $\text{mult}(P,a)=1$ just means $f(a)=0$. We will use the following simple property concerning multiplicities of composition of polynomials.

\begin{lemma}\label{lem:multComp}
Given a polynomial $P\in \F[x_1,\hdots,x_n]$ and a tuple $Q=(Q_1,\hdots,Q_n)$ of polynomials in $\F[y_1,\hdots,y_m]$, and $a\in \F^m$ we have, 
$$\text{mult}(P\circ Q, a)\ge \text{mult}(P,Q(a)).$$
\end{lemma}

The key lemma here is an extended Schwartz-Zippel bound~\cite{schwartz1979probabilistic}\cite{ZippelPaper} which leverages multiplicities.

\begin{lemma}[Schwartz-Zippel with multiplicity]\label{lem:multSchwartz}
Let $f\in \F[x_1,..,x_n]$, with $\F$ an arbitrary field, be a nonzero polynomial of degree at most $d$. Then for any finite subset $U\subseteq \F$ ,
$$\sum\limits_{a\in U^n} \text{mult}(f,a) \le d|U|^{n-1}.$$
\end{lemma}

We will also need the following lemma which lets us find polynomials which vanish on different points with differing multiplicities.

\begin{lemma}\label{lem:findPolyVanishMult}
Given a non-negative integer $d$ and a set of non-negative integers $N_x$ indexed by elements $x\in \F^n_q$ which satisfy

$$\sum\limits_{x\in \F^n_q}\binom{N_x+n-1}{n}< \binom{d+n}{n},$$
we can find a non-zero polynomial $P$ of total degree at most $d$ such that for all $x\in \F_q^n$, $P$ vanishes on $x$ with multiplicity at least $N_x$.
\end{lemma}
\begin{proof}
Note $ \binom{d+n}{n}$ is the vector space dimension of the space of polynomials in $n$ variables with total degree at most $d$. The condition of a polynomial vanishing on a point $x$ with multiplicity $N_x$ is defined by $\binom{N_x+n-1}{n}$ many linear equations in the coefficients of the polynomial. The condition of vanishing on $x$ with multiplicity $N_x$ for all $x$ is then defined by at most $\sum_{x\in \F^n_q} \binom{N_x+n-1}{n}$ many linear equations. The condition in the statement of the lemma implies that we can find a non-zero polynomial which satisfies all these conditions.
\end{proof}

\section{Entropy formulation for the Kakeya problem}\label{sec:entropy}

Let $R$ be a random variable (r.v.) taking values in $\F_q^n$. The $q$-ary {\em min entropy} of $R$ (or just min-entropy if $q$ is clear from the context) is defined as
\[ \minentq(R) = -\log_q \left( \max_{w \in \F_q^n}\bP[ R = w ] \right) \]
For example, if $R$ is distributed uniformly on a set of size $q^k$ then its min-entropy will be exactly $k$. In general, a r.v with min-entropy $k$ must have support size at least $q^k$.

We first consider a class of statements which state Furstenberg bounds in the usual manner.

\begin{definition}(Furstenberg set bound, $A(n,k)$)
Let $1 \leq k < n$ be integers. We say that the statement $A(n,k)$ holds with constant $C_{n,k}$ if the following is true: 
\begin{quote}
If $S \subset \F_q^n$ is $(k,m)$-Furstenberg then $|S| \geq C_{n,k}\cdot m^{n/k}$.	
\end{quote}
In other words $A(n,k)$ is the statement that $K(q,n,k,m)\geq C_{n,k}\cdot m^{n/k}$.
\end{definition}

Note, as mentioned earlier, the proof of the Kakeya bound in~\cite{dvir2013extensions} shows that for all $n$, $A(n,1)$ holds with $C_{n,1} = 2^{-n}$.

We now define a seemingly different statement involving min-entropy of linear maps.
\begin{definition}(Linear maps with high min-entropy, $B(n,k)$)
Let $1 \leq k < n$ be integers. We say that the statement $B(n,k)$ holds with constant $D_{n,k}$ if the following is true: 
\begin{quote}
For all $\delta \in [0,1]$, if $S \subset \F_q^n$ is of size $|S| = q^{\delta n}$ then there exists an onto linear map $\varphi : \F_q^n \mapsto \F_q^{n-k}$ such that $\minentq(\varphi(U_S)) \geq \delta(n-k) - D_{n,k}$, where $U_S$ is a random variable distributed uniformly over $S$, and $\varphi(U_S)$ is the pushforward of $U_S$.
\end{quote}
\end{definition}

In other words, $B(n,k)$ says that given the random variable $U_S$, which is uniform over a set $S$ of size $q^{\delta n}$  and hence having min-entropy $\delta n$, one can find a linear map that keeps the same {\em relative} min-entropy (the ratio between min-entropy and dimension) up to some small loss $D_{n,k}$. 

The two statements $A(n,k)$ and $B(n,k)$ are equivalent for $C_{n,k}\in (0,1]$ and $D_{n,k}\ge 0$, with a simple formula relating $C_{n,k}$ and $D_{n,k}$.

\begin{lemma}\label{lem-BtoA}
For integers $1 \leq k < n$. If  $B(n,k)$ holds  with constant $0\le D_{n,k}$, then $A(n,k)$ holds  with constant $$ C_{n,k} = q^{-\frac{n}{k}D_{n,k}}.$$
\end{lemma}
\begin{proof}

Let $S \subset \F_q^n$ be $(k,m)$-Furstenberg, and suppose that $B(n,k)$ holds.
Let $\varphi$ be an arbitrary linear map from $\F^n$ to $\F^{n-k}$.
Since $\varphi^{-1}(x)$ is a $k$-flat for each $x \in \F^{n-k}$ and $S$ is $(k,m)$-Furstenberg,
\[ \max_{x \in \F_q^{n-k}} |\varphi^{-1}(x)| \geq m, \]
and hence
\[ H_\infty^q(\varphi(U_S)) \leq -\log_q(m |S|^{-1}). \]
Taking $\delta$ such that $|S|=q^{\delta n}$, $B(n,k)$ implies that
\[ \log_q(m|S|^{-1}) \leq D_{n,k} - \delta(n-k) \]
and hence
\[ m \leq q^{D_{n,k}} |S|^{k/n}. \]
Since $A(n,k)$ is equivalent to $m \leq (|S| C_{n,k}^{-1})^{k/n}$, this implies that $A(n,k)$ holds for $C=q^{-(n/k)D_{n,k}}$, as claimed.
\end{proof}

We also show that $A(n,k)$ implies $B(n,k)$ for suitable choices of $C_{n,k}$ and $D_{n,k}$, although this direction is not needed in the proof of Theorem \ref{thm:FurstenbergRecurse}.

\begin{lemma}\label{lem-AtoB}
For integer $1 \leq k < n$. If  $A(n,k)$ holds  with constant $0<C_{n,k}\le 1$ then $B(n,k)$ holds  with constant $$ D_{n,k} = \frac{k}{n}\cdot \log_q\left(\frac{1}{C_{n,k}}\right) .$$
\end{lemma}
\begin{proof}
Let $n > k$ and suppose in contradiction that $B(n,k)$ does not hold for the above $D_{n,k}$. This means that there exists a $\delta \in [0,1]$ and a  set $S \subset \F_q^n$  of size $|S| = q^{\delta n}$ such that for any onto linear map $\varphi : \F_q^n \mapsto \F_q^{n-k}$ we have $\minentq(\varphi(U_S)) < \delta(n-k) - D_{n,k}$. By the definition of min-entropy this means that for all $\varphi$ there must exist some $v = v_{\varphi} \in \F_q^{n-k}$ such that 
\begin{equation}\label{eq-varphi-me}
	\bP\left[ \varphi(U_S) = v_\varphi \right] = \frac{ | \varphi^{-1}(v_\varphi) \cap S |}{|S|} > \frac{q^{D_{n,k}}}{q^{\delta(n-k)}}.
\end{equation}
Let $K_\varphi \subset \F_q^n$ denote the $k$-dimensional kernel of $\varphi$. Then, (\ref{eq-varphi-me}) implies that there is a shift $w_{\varphi} \in \F_q^n$ so that
\begin{equation}
	| (K_{\varphi} + w_{\varphi}) \cap S | > |S|\cdot \frac{q^{D_{n,k}}}{q^{\delta(n-k)}} \geq q^{\delta k + D_{n,k}}
\end{equation}

Since $K_\varphi$ can be any $k$-dimensional linear subspace, $S$ is $(k,m)$-Furstenberg with $m> q^{\delta k + D_{n,k}}$. Since $A(n,k)$ holds with constant $C_{n,k}$ we get that
\begin{equation}
	|S| > C_{n,k} \cdot \left( q^{\delta k + D_{n,k}} \right)^{n/k} = C_{n,k}\cdot q^{\frac{n}{k}D_{n,k}} \cdot |S|.
\end{equation}
Cancelling $|S|$ from both sides and using the expression for $D_{n,k}$, we get a contradiction.
\end{proof}

The statement $B(n,k)$ is easily generalizable, with $U_S$ replaced by a general random variable. The generalization of the statement $B(n,1)$ can be proven using a simple generalization of the proof in \cite{dvir2013extensions}. This generalized statement will allow us to perform induction to prove Furstenberg set bounds.

\begin{restatable}[Entropic-Furstenberg bound]{theorem}{entropicFurstenbergRecurse}
\label{thm:entropicFurstenbergRecurse}
For any random variable $R$ supported over $\F_q^n$ there exists an onto linear map $\phi: \F^n_q\rightarrow \F^{n-k}_q$ such that 
$$\bH^q_\infty(\phi(R))\ge \frac{n-k}{n}\bH^q_\infty(R)-\log_q(2-q^{-1})k.$$
\end{restatable}

Theorem~\ref{thm:FurstenbergRecurse} follows easily from Theorem~\ref{thm:entropicFurstenbergRecurse}.

\begin{proof}[Proof of Theorem~\ref{thm:FurstenbergRecurse}]
Theorem~\ref{thm:entropicFurstenbergRecurse} proves the statement $B(n,k)$ with constant $D_{n,k}=k\log_q(2)$. Lemma \ref{lem-BtoA} then proves Theorem \ref{thm:FurstenbergRecurse}.
\end{proof}

We will prove Theorem \ref{thm:entropicFurstenbergRecurse} using the polynomial method for the case $k=1$ and the general case will follow from an inductive argument by composing a sequence of onto maps. For that reason, we restate the $k=1$ case separately.

\begin{restatable}[Entropic bound for $k=1$]{theorem}{entropicKakeya}
\label{thm:entropicKakeya}
For any random variable $R$ supported over $\F_q^n$ there exists an onto linear map $\phi: \F^n_q\rightarrow \F^{n-1}_q$ such that 
$$\bH^q_\infty(\phi(R))\ge \frac{n-1}{n}\bH^q_\infty(R)-\log_q(2-q^{-1}).$$
\end{restatable}

\section{Proof of the entropic bound when $k=1$}\label{sec:entropykakeya}

We will prove Theorem \ref{thm:entropicKakeya} by first proving an estimate for the $\ell^n$ norm of integer valued functions over $\F^n_q$ and reducing Theorem \ref{thm:entropicKakeya} to it.

\begin{theorem}\label{thm:normBound}
Given $r\in \Z_{\ge 0}$ and a function $f:\F^n_q\rightarrow \Z$ such that for every direction $\gamma$ there exists a line $E_\gamma$ in that direction such that $\sum_{x\in E_\gamma }|f(x)| \ge r$ we have the following bound,
$$ \|f\|_{\ell^n}^n=\sum\limits_{x\in \F^n_q} |f(x)|^n\ge \frac{r^n}{\left(2-q^{-1}\right)^n}.$$
\end{theorem}
Note, if $f$ is an indicator function for a subset of $\F^n_q$ and $r=q$ then the theorem above is simply the Kakeya bound in \cite{dvir2013extensions}.
Also note that this theorem can easily be generalized to real valued functions and positive real $r$ by taking ratios and limits.

Our proof is a simple modification of the proof of the Kakeya theorem in~\cite{dvir2013extensions}. 
A more general Kakeya estimate appears in~\cite{ellenberg_oberlin_tao_2010}, but with a larger constant in place of $(2-q^{-1})^n$.

\begin{proof}[Proof of Theorem~\ref{thm:normBound}]

Fix $m$ to be a positive multiple of $r$. Let $d=mq$ where and $N=m(2q-1)/r$. It suffices to prove the following for large enough values of $m$:

\begin{align}\sum\limits_{x\in \F^n_q}\binom{N|f(x)|+n-1}{n}\ge \binom{d+n}{n}.\label{eq:keyBound}
\end{align}
Indeed, dividing by $\binom{d+n}{n}$ on both sides and substituting for $d$ and $N$ gives us

$$\sum\limits_{x\in \F^n_q} \frac{((2q-1)m|f(x)|/r+n-1)\hdots((2q-1)m|f(x)|/r)}{(mq+n)\hdots (mq+1)}\ge 1.$$
As $m$ can be arbitrarily large, we let it grow towards infinity which gives us
$$\sum\limits_{x\in \F^n_q}|f(x)|^n\ge \frac{r^n}{(2-q^{-1})^{n}},$$
which is exactly what we want to prove. Hence, we only need to prove \eqref{eq:keyBound} now.

Suppose that \eqref{eq:keyBound} is false. Using Lemma \ref{lem:findPolyVanishMult}, we can find a non-zero polynomial $P$ of total degree at most $d$ such that it vanishes on each point $x$ of $\F^n_q$ with multiplicity $N|f(x)|$.

Let $P^H$ refer to the homogenous part of $P$ of highest degree. We make the following claim.

\begin{claim}
For all $b\in \F^n_q$,
$$\text{mult}(P^H,b)\ge m.$$
\end{claim}
\begin{proof}
It is easy to see the statement is true for $b=0$ because $P^H$ is a homogenous polynomial of degree $d>m$.

Recall, for any $\alpha \in \Z_{\ge 0}^n$ its weight is defined as the sum of its coordinates. Fix any $\alpha\in \Z_{\ge 0}$ such that $\text{wt}(\alpha)=m'<m$. Let us consider $Q=P^{(\alpha)}$, that is, the $\alpha$th Hasse derivative of $P$. $Q$ has degree at most $d-m'$ and vanishes on every $x$ with multiplicity $\max(N|f(x)|-m',0)$. For any direction $b\in \F^n_q \setminus \{0\}$, we can find a point $a\in \F^n_q$ such that the line $L=\{x:x=a+bt,t\in \F^n_q\}$ satisfies
\begin{align}
    \sum_{x \in L} |f(x)| \geq r.
\end{align}
This implies
\begin{align}\label{eq:multPoly1}
\sum\limits_{x\in L} \text{mult}(Q,x)\ge \sum\limits_{x\in L} \max(N|f(x)|-m',0)\ge Nr-qm'.  \end{align}

Let $Q_{a,b}(t) = Q(a + bt)$.
Then $Q_{a,b}$ is a univariate polynomial of degree at most $d-m'$.
Lemma \ref{lem:multComp} and \eqref{eq:multPoly1} implies
\begin{align}
    \sum\limits_{t\in \F_q} \text{mult}(Q_{a,b},t)\ge \sum\limits_{x\in L} \text{mult}(Q,x)\ge Nr-qm'.\label{eq:multpoly2}
\end{align}

If $Q_{a,b}$ is non-zero then Lemma \ref{lem:multSchwartz} and \eqref{eq:multpoly2} give us the bound $Nr-qm'\le d-m'$, which implies $m(q-1)\le m'(q-1)$. This leads to a contradiction, proving that $Q(a+bt)$ is identically zero. We note $(P^H)^{(\alpha)}$ is precisely the homogenous part of highest degree of $Q$. $Q(a+bt)$ being identically zero implies $(P^H)^{(\alpha)}$ vanishes on $b$. This proves the claim.
\end{proof}

Putting everything together we now know that $P^H$, which has total degree at most $d$, vanishes on all values in $\F^n_q$ with multiplicity at least $m$. Lemma \ref{lem:multSchwartz} now implies that $mq \le d$, leading to a contradiction. This finishes the proof of the Theorem.
\end{proof}

We are now ready to prove Theorem \ref{thm:entropicKakeya}.

\begin{proof}[Proof of Theorem \ref{thm:entropicKakeya}]
We will prove this theorem for random variables $R$ such that $\bP(R=x)$ is a rational number for all $x\in \F^n_q$. After a simple limiting argument we will obtain the statement for all random variables $R$. As mentioned earlier, we will reduce to Theorem \ref{thm:normBound}. We let $\bP(R=w)=f(x)/S$ for some positive integer $S$ and non-negative integer $f(x)$ for all $x\in \F^n_q$. It is clear that $S=\sum_{x\in \F^n_q} f(x)$.

We note $\bH^q_\infty(R)$ is simply going to be $-\log_q(f(v)/S)$ where $v\in \F^n_q$ is the mode of $R$. 

Given any onto linear map $\phi:\F^n_q\rightarrow \F_q^{n-1}$, its kernel is some line passing through the origin with direction $\gamma$. 
It is easy to check that, for every $x \in \mathbb{F}_q^{n-1}$, $\bP(\phi(R)=x)$ is obtained by summing $\bP(R=y)$ over all $y$ in the line through $x$ in direction $\gamma$.

Let $\bL_\gamma$ be the set of lines in direction $\gamma$. This means we can write $\bH^q_\infty(\phi(R))$ as
$$\bH^q_\infty(\phi(R))=-\log_q\left(\max\limits_{\ell \in \bL_\gamma} \sum\limits_{x\in \ell} \bP(R=x)\right).$$

We now pick the $\phi$ for which $\bH^q_\infty(\phi(R))$ is the largest. This is basically done by picking the direction $\gamma$ such that $\max_{\ell \in \bL_\gamma} \sum\limits_{x\in \ell} \bP(R=x)$ is the smallest. 
Let $\gamma_0$ be that direction and $\max_{\ell \in \bL_{\gamma_0}} \sum\limits_{x \in \ell} \bP(R=x)$ equals $r/S$ where $r$ is some non-negative integer. We can now re-write the statement of the Theorem as follows:

\begin{align*}
    &-\log_q\left(\frac{r}{S}\right)\ge -\frac{n-1}{n}\log_q\left(\frac{f(v)}{S}\right)-\log_q(2-q^{-1})\\
    \iff& \frac{S}{r}\ge \frac{1}{2-q^{-1}}\left(\frac{S}{f(v)}\right)^{1-1/n}\\
    \iff& \left(\frac{S}{f(v)}\right)^{1/n}\ge   \frac{1}{2-q^{-1}}\frac{r}{f(v)}\\
    \iff& \sum\limits_{x\in \F^n_q} f(x)f(v)^{n-1}\ge   \frac{1}{(2-q^{-1})^n}r^n\numberthis \label{eq:KakeyaMult}
\end{align*}

Noting that $f(v)\ge f(x)\ge 0$ for all $x$, \eqref{eq:KakeyaMult} immediately follows from Theorem \ref{thm:normBound}.
\end{proof}

\section{Proving the general entropic bound}\label{sec:entorpyfurst}

Let us first prove Theorem \ref{thm:entropicFurstenbergRecurse} which is obtained from Theorem \ref{thm:entropicKakeya} by a simple recursion.

\begin{proof}[Proof of Theorem \ref{thm:entropicFurstenbergRecurse}]
We induct over $k$. Theorem \ref{thm:entropicKakeya} is precisely the case $k=1$. Now, let it be true for some fixed $k$. This means given any random variable $R$ supported over $\F_q^n$ we can find an onto random variable $\phi:\F_q^n\rightarrow \F_q^{n-k}$ such that,
\begin{align}
    \bH^q_\infty(\phi(R))\ge \frac{n-k}{n}\bH^q_\infty(R)-\log_q(2-q^{-1})k.\label{eq:entrkInduct}
\end{align}
Applying Theorem \ref{thm:entropicKakeya} on $\phi(R)$ we can find another onto function $\psi:\F_q^{n-k}\rightarrow \F_q^{n-k-1}$ such that,
\begin{align}
    \bH^q_\infty(\psi(\phi(R))) \ge \frac{n-k-1}{n-k}\bH^q_\infty(\phi(R))-\log_q(2-q^{-1}).\label{eq:entrk1}
\end{align}
Substituting \eqref{eq:entrkInduct} in \eqref{eq:entrk1} proves the required statement.
\end{proof}

\section{Better bounds when $n$ is divisible by $k$}

In this section we will prove Theorem \ref{thm:FurstenbergDiv} which gives us much better bounds in the case when $n$ is divisible by $k$.

\begin{proof}[Proof of Theorem \ref{thm:FurstenbergDiv}]\label{sec:divide}
As $k$ is a factor of $n$ we can find a positive integer $r$ such that $n=rk$. Note there exists an $\F_q$-linear isomorphism between $\F_q^{n}$ and $\F_{q^k}^{r}$. This quickly follows from the fact $\F_{q^k}$ is by definition $\F_q[x]/I$ where $I$ is a principal ideal generated by a degree $k$ irreducible polynomial in $\F_q[x]$. This allows us to treat a point set $S$ in $\F_q^n$ as a point set in $\F_{q^k}^r$. It is easy to see that any line in $\F_{q^k}^r$ is a $k$-dimensional subspace in $\F_q^n$. This means $S$ is a Kakeya set in $\F_{q^k}^r$. Using the Kakeya bound \eqref{eqn:generalMKakeya} we have,
$$|S|\ge \frac{1}{2^{n/k}}m^{n/k},$$
which is precisely what we wanted.
\end{proof}

One could use a similar argument to prove bounds in the style of Theorem \ref{thm:entropicFurstenbergRecurse} with better constants. In fact, when $n-k$ has a factor smaller than $k$ we can combine the recursive argument of Theorem \ref{thm:entropicFurstenbergRecurse} and argument presented in this section to obtain slightly better constants for Furstenberg set bounds.

\section{Proof of Theorems \ref{th:largeMFurst} and \ref{th:pureIncidences}}\label{sec:geometric}
	
	We start by proving three lemmas.
	The proof of Theorem \ref{th:pureIncidences} depends only on Lemma \ref{th:heavyFlatsCoverManyPoints}.
	The other two lemmas are only needed in the proof of Theorem \ref{th:largeMFurst}.
	
	The first lemma shows that a set of flats witnessing a Furstenberg set contains many flats of lower dimension.
	
	\begin{lemma}\label{th:kakeyaForFlats}
		Let $F$ be a set of $k$-flats in $\F_q^n$, one parallel to each rank $k$ subspace, with $2 \leq k < n$.
		Let $1 \leq \ell < k$.
		The number of $\ell$-flats contained in the flats of $F$ is at least $\binom{n}{\ell}_q K(q,n-\ell,k-\ell,q^{k-\ell})$.
	\end{lemma}
	\begin{proof}
	    The basic observation behind this lemma is that the $\ell$-flats that are contained in flats of $F$ and are parallel to a fixed rank $\ell$ subspace correspond to the points of a $(k-\ell, q^{k-\ell})$-Furstenberg set in $\mathbb{F}_q^{n-\ell}$.
	    The bound in the conclusion of the lemma comes from summing over all rank $\ell$ subspaces of $\F_q^n$.
	    
		For each rank $\ell$ subspace $\Lambda$, choose a rank $n-\ell$ subspace $P_\Lambda$ so that $\Lambda \cap P_\Lambda$ is the origin.
        Since $\dim(\Lambda \cap P_\Lambda) = 0$, equation (\ref{eq:dimSpan}) implies that $\overline{\Lambda,P_\Lambda} = \F_q^n$.
		Let $F_\Lambda \subset F$ be the set of flats of $F$ that contain a translate of $\Lambda$.
		We will show that $K_\Lambda = \bigcup_{\Gamma \in F_{\Lambda}} (\Gamma \cap P_{\Lambda})$ is a $(k-\ell, q^{k-\ell})$-Furstenberg set in $P_\Lambda$.
		
		Let $g$ be the map from $k$-dimensional subspaces of $\F_q^n$ that contain $\Lambda$ to $(k-\ell)$-dimensional subspaces of $P_\Lambda$ defined by $g(\Gamma) = P_\Lambda \cap \Gamma$.
		Since $\overline{\Gamma, P_\Lambda} = \F_q^n$ for any subspace $\Gamma$ that contains $\Lambda$, (\ref{eq:dimSpan}) implies that $g$ is well-defined.
		In addition, any rank $k-\ell$ subspace $H$ contained in $P_\Lambda$ intersects $\Lambda$ only at the origin, so $\dim(\overline{\Lambda, H}) = k$.
		Consequently, $g$ is bijective.
		
		Let $v \in \F_q^n$ be arbitrary.
		Let $v_\Lambda$ and $v_{P_\Lambda}$ so that $v = v_\Lambda + v_{P_\Lambda}$, where $v_\Lambda \in \Lambda$ and $v_{P_\Lambda} \in P_\Lambda$.
		Since $\overline{\Lambda,P_\Lambda} = \F_q^n$, this is always possible.
		Let $\Gamma$ be a rank $k$ subspace that contains $\Lambda$.
		Then,\[
		    (\Gamma + v) \cap P_\Lambda = (\Gamma + v_{P_\Lambda})\cap P_\lambda 
		    =(\Gamma + v_{P_\Lambda}) \cap (P_\Lambda + v_{P_\Lambda})
		    = \Gamma \cap P_\Lambda + v_{P_\Lambda}.\]
		
		We are now ready to show that $K_\Lambda$ is a $(k-\ell, q^{k-\ell})$-Furstenberg set.
		Let $H$ be a $(k-\ell)$-dimensional subspace contained in $P_\Lambda$.
		By the hypothesis on $F$, there is $v \in \F_q^n$ such that $g^{-1}(H) + v \in F_\Lambda$.
		Hence, $H + v_{P_\Lambda} \subseteq K_\Lambda$.
		
		By definition, $|K_\Lambda| \geq K(q,n-\ell,k-\ell, q^{k-\ell})$.
		Each point in $K_\Lambda$ is the intersection of $P_\Lambda$ with an $\ell$-flat parallel to $\Lambda$ that is contained in some flat of $F$.
		So, the set $L_\Lambda$ of $\ell$-flats parallel to $\Lambda$ and contained in $k$-flats of $F$ is in $1$-$1$ correspondence with the set $K_\Lambda$.
		Hence,
		\[ \sum_{\Lambda} |L_\Lambda| = \sum_{\Lambda} |K_\Lambda| \geq \binom{n}{\ell}_q K(q,n-\ell,k-\ell,q^{k-\ell}), \]
		where $\Lambda$ ranges over all rank $\ell$ subspaces of $\F_q^n$.
	\end{proof}

	For the proof of Theorem \ref{th:largeMFurst}, we only need the case $\ell = k-1$ of Lemma \ref{th:kakeyaForFlats}.
	The application of (\ref{eq:kakeya}) to obtain an explicit bound on $K(q,n,1,q)$ for use with Lemma \ref{th:kakeyaForFlats} is the only application in this section of any result proved using the polynomial method.

	\begin{lemma}\label{th:kakeyaBecks}
	    Let $2 \leq k < n$. 
		Let $S$ be a $(k,m)$-Furstenberg set in $\F_q^n$.
		Let $\delta < 1$.
		Let $G_r$ be the set of $(k-1)$-flats that are each incident to at least $r = \delta mq^{-1} + 1$ points of $S$.
		If $m \geq 2^{n+3-k} q (1-\delta)^{-2}$, then $|G_r| > 2^{k-2-n}q^{n-k+1} \binom{n}{k-1}_q$.
	\end{lemma}
	\begin{proof}
	    Let $F$ be a set of $k$-flats that each intersect $S$ in at least $m$ points, such that, for each rank $k$ subspace, there exists a flat of $F$ parallel to it.
	    
		By Lemma \ref{th:kakeyaForFlats} and the Kakeya bound (\ref{eq:kakeya}), there is a set $G$ of $(k-1)$-flats contained in the flats of $F$ with 
		\[|G| \geq  K(q,n-k+1,1,q) \binom{n}{k-1}_q \geq 2^{k-1-n} q^{n-k+1} \binom{n}{k-1}_q.\]
		
		Let $G_p \subseteq G$ be those flats of $G$ that are $(S,r)$-poor.
		We will show that $|G_p| < 2^{-1}|G|$, which implies the conclusion of the lemma.
		
		Applying Lemma \ref{th:becks}, the number of $(S,r)$-poor $(k-1)$-flats contained in any given $k$-flat is at most $(1+mq^{-1}(1-\delta)^2)^{-1}q(1-q^k)(1-q)^{-1}$.
		Since $m \geq 2^{n+3-k}q(1-\delta)^{-2}$, we have \[(1+mq^{-1}(1-\delta)^2)^{-1} < 2^{k-3-n}.\]
		
		Summing over the flats of $F$ and using the exact expression (\ref{eq:exactQBinom}) for $q$-binomial coefficients,
		\begin{align*}
		|G_p| &\leq  |F| (1+mq^{-1}(1-\delta)^2)^{-1} \frac{1-q^k}{1-q} q 
		\\ &< 2^{k-3-n} |F| \frac{1-q^k}{1-q} q \\
		&= 2^{k-3-n} \binom{n}{k}_q \frac{1-q^k}{1-q} q \\
		&= 2^{k-3-n} \binom{n}{k-1}_q \frac{1-q^{n-k+1}}{1-q} q \\
		&< 2^{k-2-n} \binom{n}{k-1}_q q^{n-k+1} \\
		&\leq 2^{-1}|G|,
		\end{align*}
		as claimed.
	\end{proof}
	
	The next lemma is essentially a reformulation of Lemma \ref{th:incidenceBound}.

	\begin{lemma}\label{th:heavyFlatsCoverManyPoints}
		Let $P \subseteq \F_q^n$ be a set of points.
		Let $\delta, \gamma > 0$, and let $L$ be a set of $\ell$-flats that each contain at least $\delta q^\ell$ points of $P$, and suppose that $|L| = \gamma q^{n-\ell} \binom{n}{\ell}_q$.
		Let $\kappa = \gamma q^\ell$.
		Then,
		\[|P| \geq \left ( \delta \kappa (\kappa+1)^{-1} - \sqrt{\delta(1-\delta) \kappa^{-1}} \right ) q^n. \]
	\end{lemma}
	\begin{proof}
		Let $\eps = |P| q^{-n}$.
		If $\delta \leq \eps$, then $|P| \geq \delta q^n$, which is stronger than the conclusion of the lemma.
		Hence, we may assume that $\eps < \delta$.
		
		Since each flat of $L$ contains at least $\delta q^\ell$ points of $P$, it follows that $I(P,L) \geq \delta q^\ell |L|$.
		By Lemma \ref{th:incidenceBound},
		\[\delta q^\ell |L| \leq \eps q^\ell |L| + \sqrt{q^\ell \binom{n-1}{\ell}_q |P| \, |L| \left (1 - q^{-n}|P| \right )}. \]	
		Rearranging,
		\[(\delta - \eps)^2 q^\ell |L| \leq \eps q^n (1-\eps) \binom{n-1}{\ell}_q. \]
		Since $\binom{n}{\ell}_q > q^\ell \binom{n-1}{\ell}_q$, applying the hypothesis on $|L|$ gives
		\begin{equation}\label{eq:quad}(\delta - \eps)^2 q^\ell \gamma - \eps(1-\eps) < 0. \end{equation}
		Since the coefficient of $\eps^2$ in (\ref{eq:quad}) is positive, $\eps$ must be greater than the smaller root of (\ref{eq:quad}).
		Hence,
		\begin{align*}
		\eps &> \frac{1 + 2 \delta \kappa - \sqrt{(2\delta \kappa + 1)^2 - 4(\kappa + 1) \delta^2 \kappa}}{2(\kappa + 1)} \\
		&= \frac{1 + 2 \delta \kappa - \sqrt{1 + 4 \delta \kappa (1 - \delta)}}{2(\kappa + 1)}\\
		&> \frac{\delta \kappa - \sqrt{\delta \kappa (1 - \delta)}}{\kappa + 1} \\
		&> \delta \kappa (\kappa+1)^{-1} - \sqrt{\delta(1-\delta) \kappa^{-1}}.
		\end{align*}
	\end{proof}

    We are now ready to prove Theorems \ref{th:largeMFurst} and \ref{th:pureIncidences}.
    
	\begin{proof}[Proof of Theorem \ref{th:pureIncidences}]
	    Applying Lemma \ref{th:heavyFlatsCoverManyPoints} with $\delta = mq^{-k}$ and $\gamma=q^{k-n}$ yields
	    \begin{align*}
	    |S| &\geq mq^{n-k} \left( 1 - (q^{2k-n} +1)^{-1} - \sqrt{(1-mq^{-k})q^{n-k}m^{-1}} \right) \\
	    &\geq mq^{n-k} \left( 1 - q^{n-2k} - q^{2n-4k} - \sqrt{q^{n-k}m^{-1}} + \sqrt{q^{n-2k}} \right).
	    \end{align*}
	    
	   The assumption that $k > n/2$ implies that $q^{(n-2k)/2} > q^{2(n-2k)}$, hence
	   \[ |S| \geq mq^{n-k} \left( 1 - q^{n-2k} - \sqrt{q^{n-k}m^{-1}}\right).\]
	   
	\end{proof}
	
	\begin{proof}[Proof of Theorem \ref{th:largeMFurst}]
	    Let $S$ be a $(k,m)$-Furstenberg set in $\mathbb{F}_q^n$ with $2 \leq k < m$ and $2^{n+7-k}q\eps^{-2} m \leq q^k$.
	    We show that $|S| \geq (1-\eps)mq^{n-k}$.
	    
	    Apply Lemma \ref{th:kakeyaBecks} to $S$ with $\delta = 1-\eps/4$.
	    This gives a set $G_r$ of $(k-1)$-flats, each incident to more than $(1-\eps/4)mq^{-1}$ points of $S$, with $|G_r| > 2^{k-2-n}q^{n-k+1}\binom{n}{k-1}_q$.
		
		Next apply Lemma \ref{th:heavyFlatsCoverManyPoints} to $G_r$ with $\delta = (1-\eps/4)mq^{-k}$, $\ell = k-1$, and $\gamma = 2^{k-2-n}$.
		As in Lemma \ref{th:heavyFlatsCoverManyPoints}, let $\kappa = \gamma q^{k-1}$.
		Note that $q^k \geq m \geq 2^{n+7-k} q \eps^{-2}$, and hence 
		\begin{align*} \kappa (1 + \kappa)^{-1} &\geq 1- \eps^2 2^{-5} > 1-\eps/4, \text{ and} \\ 
		\kappa^{-1} &\leq 2^{-5} \eps^2.
		\end{align*}
		Thus we have
		\begin{align*}
			|S|q^{-n} &\geq \delta \kappa (\kappa+1)^{-1} - \sqrt{\delta(1-\delta) \kappa^{-1}} \\
			&> \delta (1-\eps/4) - \sqrt{\delta} (\eps / 4) \\
			&> \delta(1-\eps/2) \\
			&= (1 - \eps/4)(1-\eps/2)mq^{-k} \\
			&> (1-\eps)mq^{-k}.
		\end{align*}
	\end{proof}


\section*{Acknowledgments} 
The authors are grateful to the anonymous reviewer for numerous helpful comments.

\bibliographystyle{amsplain}


\begin{dajauthors}
\begin{authorinfo}[dhar]
  Manik Dhar\\
  Department of Computer Science\\ Princeton University\\
  Princeton, New Jersey, USA\\
  \texttt{manikd@princeton.edu}
\end{authorinfo}
\begin{authorinfo}[dvir]
  Zeev Dvir\\
  Department of Computer Science and Department of Mathematics\\ 
  Princeton University\\
  Princeton, New Jersey, USA\\
  \texttt{zeev.dvir@gmail.com}
\end{authorinfo}
\begin{authorinfo}[lund]
  Ben Lund\\
  Discrete Mathematics Group\\
  Institute for Basic Science\\
  Daejeon, South Korea\\
  \texttt{lund.ben@gmail.com}
\end{authorinfo}
\end{dajauthors}

\end{document}